\documentclass{svproc}
\usepackage[utf8]{inputenc}
\usepackage[T2A]{fontenc}
\usepackage[english]{babel}

\usepackage{graphicx} 

\usepackage{url}

\begin{document}
\mainmatter              
\title{A cycle with a ``short'' period in the phenomenological model of a single neuron}
\titlerunning{A cycle with a 'short' period in the phenomenological model of a single neuron}
%
\author{Margarita M. Preobrazhenskaia, 
Igor E. Preobrazhenskii, 
}
\authorrunning{M.M. Preobrazhenskaia, 
I.E. Preobrazhenskii}
%
\tocauthor{Margarita M. Preobrazhenskaia, 
Igor E. Preobrazhenskii}
\institute{
Centre of Integrable Systems, P.G. Demidov Yaroslavl State University,\\ Yaroslavl, Russia,\\ 
\email{rita.preo@gmail.com, preobrazenskii@gmail.com}
}

\maketitle              

\begin{abstract} We consider the relay version of the generalized Hutchinson's equation as a phenomenological model of an isolated neuron. After an exponential substitution, the equation takes the form of a differential-difference equation with a piecewise-constant right-hand side. In the work of A. Yu. Kolesov and others, this equation was studied with a negative continuous initial function, and the existence and orbital stability of a periodic solution with a period longer than the delay were proven. In the present work, all possible solutions with continuous on the interval of the delay length initial functions, containing no more than two zeros, are constructed. It is proven that the equation has a periodically unstable solution, the period of which is shorter than the delay.

\end{abstract}

\section{Introduction}
The generalized Hutchinson equation or its relay version has been used as the basis for neural models in a number of works \cite{GlyKolRoz2013_,umn,Glyzin2013a,GlyKolRoz2013,Preob2018,GlyPre2019_r,GlyPre2019,preob2022,perc,pnd_vera}, where various neural systems are modeled using differential equations with delays. By the generalized Hutchinson equation, we understand the equation \cite{hutch}:
\begin{equation}
    \label{eq_hutch}
    \dot{u}=\lambda f(u(t-1))u,
\end{equation}
where nonlinearity is determined by a fairly smooth function $f$, which has a positive value at zero and tends towards a negative constant at large argument values:  
$$f(0)=1,\ f(u)\rightarrow - a  \mbox{ as }u\rightarrow +\infty, \ a >0.$$
It is worth noting that (\ref{eq_hutch}) is not the only way to interpret the generalization of the Hutchinson equation. For instance, other generalizations are considered in works like \cite{Gly2007,GlyKolRoz2009,Kasch_hutch}.

As the relay version of equation (\ref{eq_hutch}), we consider the following~\cite{GlyPre2019_r}:
\begin{equation}
    \label{eq_hutch_discr}
    \dot{u}=\lambda {F}(u(t-1))u,
\end{equation}
where instead of $f$, a piecewise-constant function 
\begin{equation}
    \label{F}
F(u)\stackrel{\rm def}{=}\left\lbrace 
	\begin {array}{cl} 
1,& 0<u\leq1,	
	\cr 
- a ,& u>1,
	\end {array}\right. 
\end{equation}
is used. This function switches values when the argument crosses $1$ and takes on the same values as $f$ at $0$ and $+\infty$. In both equations (\ref{eq_hutch}) and (\ref{eq_hutch_discr}), the scalar function $u(t)>0$ represents the normalized membrane potential, and $\lambda>0$ is a rate of electrical processes in a nerve cell.

Equations (\ref{eq_hutch}) and (\ref{eq_hutch_discr}) can be considered as phenomenological models of an isolated neuron.

A well-known result \cite{hutch} is as follows: in the case $\lambda\gg1$, equation (\ref{eq_hutch}) has an asymptotically stable relaxation \cite{GlyKolRoz2011a,GlyKolRoz2011b,GlyKolRoz2011c} cycle:
$$v_0(t)={\rm e}^{\lambda (x_0(t)+O(1/\lambda))} \mbox{ as } \lambda \rightarrow +\infty,$$
where $x_0(t)$ is a periodic function that changes sign over one period length, see Fig. \ref{pic:x0}.
\begin{figure}[h!]
\centering
\begin{minipage}{.49\linewidth}
  \centering\includegraphics[width=12pc]{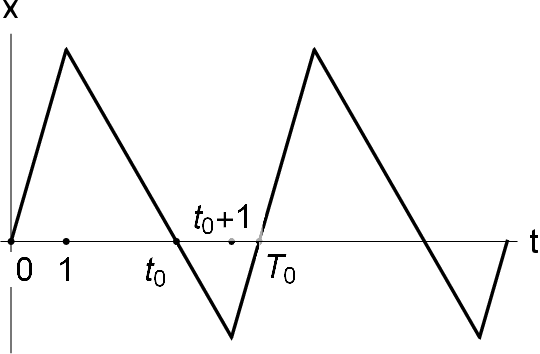}
 \caption{The function $x_0(t)$.}
\label{pic:x0}
\end{minipage}
\begin{minipage}{.5\linewidth}
  \centering\includegraphics[width=12pc]{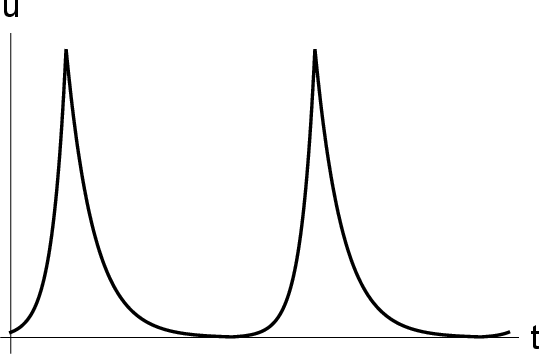}
 \caption{The function $u_0(t)$.}
\label{pic:u0}
\end{minipage}
\end{figure}
The same result holds for equation (\ref{eq_hutch_discr}) (without asymptotic approximations). For any positive $\lambda$, equation (\ref{eq_hutch_discr}) has an orbitally stable cycle (see \cite{GlyPre2019_r}):
$$u_0(t)={\rm e}^{\lambda x_0(t)}.$$
Due to the exponential dependence on $x_0$, the functions $u_0$ and $v_0$ exhibit high spikes per the period (see Fig. \ref{pic:u0}), which is significant in neural modeling. 

Within this work, we focus on investigating relay equation (\ref{eq_hutch_discr}). Our aim is to find periodic solutions of  equation (\ref{eq_hutch_discr}) that are different from the known solution $u_0$. It is convenient to study this equation after the exponential substitution $u={\rm e}^{\lambda x}$, which transforms it to the form
\begin{equation}
    \label{eq_rele}
    \dot{x}=R(x(t-1)),
\end{equation}
where 
\begin{equation}
    \label{R}
R(x)\stackrel{\rm def}{=}\left\lbrace 
	\begin {array}{cl} 
1,& x\leq0,	
	\cr 
- a ,& x>0.
	\end {array}\right. 
\end{equation}

In \cite{hutch}, continuous functions negative on the interval of time delay with a value of zero at zero were considered as the set of initial functions for equation (\ref{eq_rele}). However, we search for solutions of equation (\ref{eq_rele}) with a continuous initial function on the interval of time delay, containing no more than two zeros on this interval.
 
Let's outline the plan of the article.
In Section \ref{sec_1root}, it is described that if a sign-constant initial function or a function containing one zero is chosen for equation (\ref{eq_rele}), only the well-known solution from the work \cite{hutch}, shifted in time, can be obtained. In Section \ref{sec_2root}, a set of arbitrary continuous functions with two zeros on the interval $[-1,0)$ taking the value zero at zero is considered as the initial function set for equation (\ref{eq_hutch}). It is shown (Theorem 1) that for each fixed parameter $a>0$, there exist unique values of this pair of zeros, such that the corresponding solution of equation (\ref{eq_rele}) is periodic and distinct from the known solution $x_0(t)$. However, this solution turns out to be unstable (Theorem 2). With a different choice of two zeros of the initial function, it converges to the $x_0$ regime over time.

\section{Solution of equation (\ref{eq_rele}) with a continuous initial function containing no more than one zero}\label{sec_1root}

In the work \cite{hutch}, the initial set (see Fig. \ref{pic:phi1})
\begin{figure}[h!]
\centering\includegraphics[width=12pc]{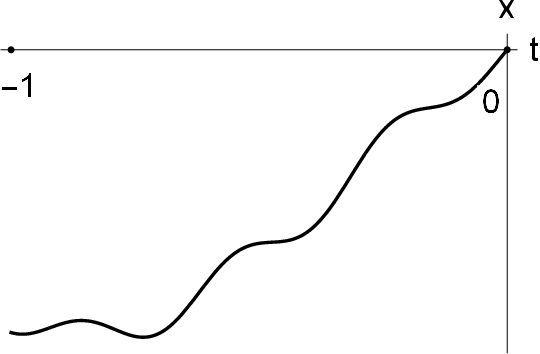}
 \caption{The function $\varphi(t)$ from the set (\ref{init_set1}).}
\label{pic:phi1}
\end{figure} 
\begin{equation}
    \label{init_set1}
    \varphi\in C[-1,0],\ \varphi(t)<0\mbox{ as } t\in[-1,0),\ \varphi(0)=0,
\end{equation}
 is introduced for equation (\ref{eq_rele}) and the following statement is proved.
\begin{proposition}\label{pr_hutch}
    Equation (\ref{eq_rele}) with an initial function from the set (\ref{init_set1}) has an orbitally stable periodic solution
\begin{equation}
    \label{x0}
x_0(t)\stackrel{\rm def}{=}\left\lbrace\begin{array}{cl}
t,& t\in[0, 1],\\
- a (t-t_0),& t\in [1, t_0+1],\\
t-T_0,& t\in [t_0+1,T_0],
\end{array}\right. \quad x_0(t+T_0)\equiv x_0(t),
\end{equation}
\begin{equation}
    \label{T_0}
   t_0\stackrel{\rm def}{=}( a +1)/ a , \quad T_0\stackrel{\rm def}{=}( a +1)^2/ a .
\end{equation}
\end{proposition}

Note that a similar statement holds true if we replace (\ref{init_set1}) with
\begin{equation}
    \label{init_set2}
    \varphi\in C[-1,0],\ \varphi(t)<0\mbox{ as } t\in[-1,0),\ \varphi(0)=-d,
\end{equation}
or
\begin{equation}
    \label{init_set3}
    \varphi\in C[-1,0],\ \varphi(t)>0\mbox{ as } t\in[-1,0),\ \varphi(0)=d,
\end{equation}
where $d\geq0$ is a fixed parameter. Typical representatives of these sets are depicted in Figures \ref{pic:phi2} and \ref{pic:phi3}.
\begin{figure}[h!]
\centering
\begin{minipage}{.49\linewidth}
  \centering\includegraphics[width=12pc]{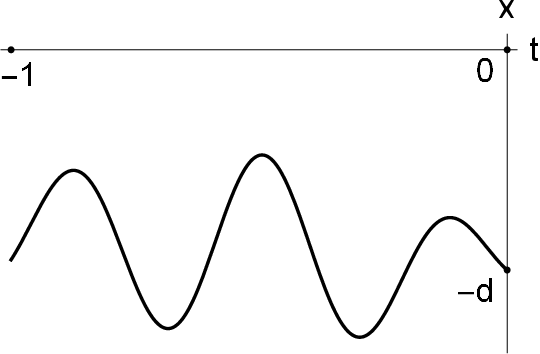}
 \caption{The function $\varphi(t)$ from the set (\ref{init_set2}).}
\label{pic:phi2}
\end{minipage}
\begin{minipage}{.5\linewidth}
  \centering\includegraphics[width=12pc]{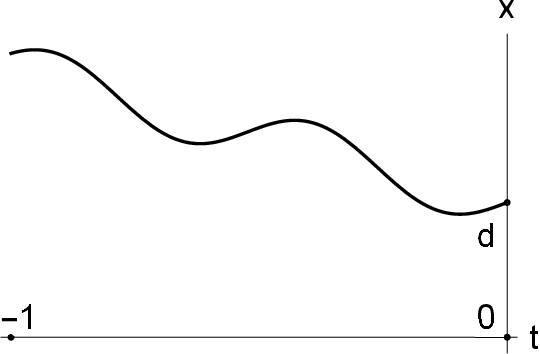}
 \caption{The function $\varphi(t)$ from the set (\ref{init_set3}).}
\label{pic:phi3}
\end{minipage}
\end{figure}

\begin{proposition}\label{pr_hutch_com}
    Equation (\ref{eq_rele}) with an initial function from the set (\ref{init_set2}) or (\ref{init_set3}) has an orbitally stable periodic solution $x_0(t-d)$, $t\geq d$.
\end{proposition}
The proof of this fact is carried out through a step-by-step method similar to what was done in \cite{hutch}.

It turns out that if we introduce an initial set with one root on a segment of length 1, then the solution $x_0$ also establishes itself starting from a certain point in time.

Let's fix the parameters $\tau\in(0,1)$ and $d>0$.

\begin{figure}[h!]
\centering
\begin{minipage}{.49\linewidth}
  \centering\includegraphics[width=12pc]{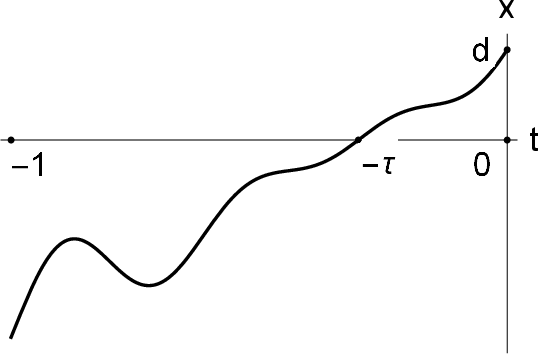}
 \caption{The function $\varphi(t)$ from the set (\ref{init_set4}).}
\label{pic:phi4}
\end{minipage}
\begin{minipage}{.49\linewidth}
  \centering\includegraphics[width=12pc]{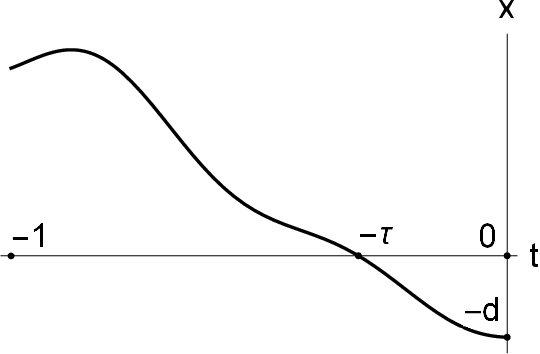}
 \caption{The function $\varphi(t)$ from the set (\ref{init_set5}).}
\label{pic:phi5}
\end{minipage}
\end{figure}

\begin{proposition}\label{pr_1r1}
    Equation (\ref{eq_rele}) with an initial function from the set 
\begin{eqnarray}
    \label{init_set4}
    \varphi\in C[-1,0],\ \varphi(t)<0&\mbox{ as }& t\in[-1,-\tau),\ \varphi(-\tau)=0,\nonumber
    \\
    \varphi(t)>0&\mbox{ as }& t\in(-\tau,0],\ \varphi(0)=d,
\end{eqnarray}
has an orbitally stable periodic solution $x_0(t-(1-\tau)t_0-d/a-a-1)$ for $t\geq (1-\tau)t_0+d/a$.
\end{proposition}

\begin{proposition}\label{pr_1r2}
    Equation (\ref{eq_rele}) with an initial function from the set 
\begin{eqnarray}
    \label{init_set5}
    \varphi\in C[-1,0],\ \varphi(t)>0&\mbox{ as }& t\in[-1,-\tau),\ \varphi(-\tau)=0,\nonumber
    \\
    \varphi(t)<0&\mbox{ as }& t\in(-\tau,0],\ \varphi(0)=-d,
\end{eqnarray}
has an orbitally stable periodic solution $x_0(t-(1-\tau)t_0-d)$ for $t\geq (1-\tau)t_0+d$.
\end{proposition}

These statements follow from Proposition \ref{pr_hutch_com}.

\section{ Solutions of equation (\ref{eq_rele}) with a continuous initial function having two zeros} \label{sec_2root}

Now, let's fix two parameters $0<\tau<\theta<1$ and consider the following
 \begin{eqnarray}
    \label{init_set6}
    \varphi\in C[-1,0],\ \varphi(t)<0&\mbox{ as }& t\in[-1,-\theta),\ \varphi(-\theta)=0,\nonumber
    \\
    \varphi(t)>0&\mbox{ as }& t\in(-\theta,-\tau),\ \varphi(-\tau)=0,\nonumber
    \\
    \varphi(t)<0&\mbox{ as }& t\in(-\tau,0),\ \varphi(0)=0,
\end{eqnarray}
 as the set of initial functions (see Fig. \ref{pic:phi6})
\begin{figure}[h!]
\centering\includegraphics[width=12pc]{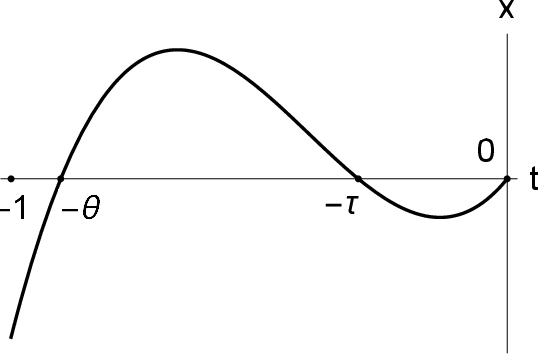}
 \caption{The function $\varphi(t)$ from the set (\ref{init_set6}).}
\label{pic:phi6}
\end{figure} 

\begin{theorem}\label{th_2r_per}
Let 
\begin{equation}
    \label{fix_point}
    \tilde{\theta}=\frac{(a+1)^2}{a^2+3a+1},\quad \tilde{\tau}=\frac{a(a+1)}{a^2+3a+1}.
\end{equation}
Equation (\ref{eq_rele}) with an initial function from the set (\ref{init_set5}) for $(\theta,\tau)=(\tilde{\tau},\tilde{\theta})$ has a $\theta$-periodic solution (see Fig. \ref{pic:y0}) defined by:
\begin{equation}
    \label{y0}
    y_0(t) \stackrel{\rm def}{=} 
    \left\lbrace
    \begin{array}{cl}
        t, & t\in [0, 1-\tilde{\theta}],\\
        - a (t-(1-\tilde{\theta})t_0), & t\in [1-\tilde{\theta},1-\tilde{\tau}],\\
        t-\tilde{\theta}, & t\in [1-\tilde{\tau},\tilde{\theta}],
    \end{array}
    \right.
    \quad y_0(t+\tilde{\theta}) \equiv y_0(t).
\end{equation}
For $0<\tau<\theta<1$, not equal to $\tilde{\tau}$, $\tilde{\theta}$ respectively, there exists $\tilde{t} \geq 0$ such that the solution of equation (\ref{eq_rele}) with the initial function (\ref{init_set6}) coincides with the function $x_0(t-\tilde{t})$ for $t \geq \tilde{t}.$
\end{theorem}
\begin{figure}[h!]
\centering\includegraphics[width=20pc]{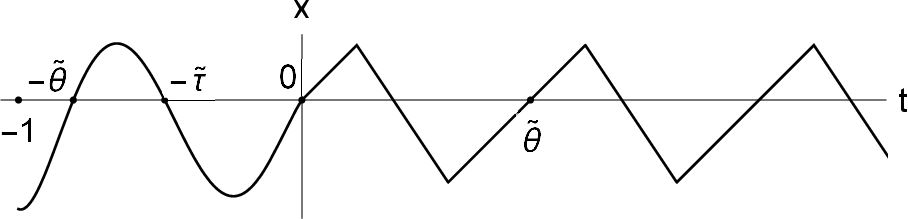}
 \caption{The solution (\ref{y0}).}
\label{pic:y0}
\end{figure} 
\begin{proof}
    Let's denote by $x_\varphi(t)$ the solution of equation (\ref{eq_rele}) with the initial function (\ref{init_set6}). Let's plot $x_\varphi(t)$ over the interval $[0,1]$. Since the initial function $\varphi$ from the selected class (\ref{init_set6}) is negative on the intervals $[-1,-\theta)$, $(-\tau,0)$ and positive on $(-\theta,-\tau)$, then for $t\in[0,1-\theta)$, $t\in(1-\tau,1)$, the solution coincides with affine functions with slope $1$ and with slope $-a$ for $t\in(1-\theta,1-\tau)$. By considering the continuity of the solution, we can derive the following:
$$
    x_\varphi(t)=\left\lbrace\begin{array}{cl}
t,& t\in[0, 1-\theta],\\
- a (t-(1-\theta)t_0),& t\in [1-\theta,1-\tau],\\
t-(a+1)(\theta-\tau),& t\in [1-\tau,1].
\end{array}\right.
$$
Moreover, if the value at $1-\tau$ turns out to be positive or the value at $1$ is negative, a interval of length 1 is formed where $x_\varphi(t)$ keeps a sign. Therefore, according to Proposition \ref{pr_hutch_com}, at certain time, the regime $x_0$ settles, shifted in time. Henceforth, we consider that $x_\varphi(1-\tau)<0$ and $x_\varphi(1)>0$, which is equivalent to the double inequality:
\begin{equation}
    \label{neq_theta_tau}
    \frac{a\tau+1}{a+1}<\theta<\frac{(a+1)\tau+1}{a+1}.
\end{equation}
With such constraints on the intervals $(1-\theta,1-\tau)$ and $(1-\tau,1)$, there exist respectively roots
\begin{equation}
\label{t1t2}
    t_1=(1-\theta)t_0 \mbox{ and } t_2=(a+1)(\theta-\tau)
\end{equation}
of equation $x_\varphi(t)=0$. 

Let's define the Poincaré mapping operator $\Pi:C[-1,0]\to C[-1,0]$ by the formula
$$\Pi\varphi(t)=x_\varphi(t+t_2) \mbox{ as } t\in[-1,0].$$
The existence of a periodic solution of equation (\ref{eq_rele}) is related to the existence of a fixed point of the operator $\Pi$.
From $\Pi$, a defining two-dimensional operator emerges, mapping a pair of numbers $(\theta, \tau)$ to a pair $(t_2, t_2-t_1)$ corresponding to the roots of the equation $x_\varphi(t+t_2)=0$ on the interval $[-1,0]$. Taking (\ref{t1t2}) into account, this mapping is described by the formula
\begin{equation}
\label{Phi}
\Phi\left(\begin{array}{c}\theta\\\tau\end{array}\right)=
(a+1)\left(\begin{array}{cc}1&-1\\t_0&-1\end{array}\right)
\left(\begin{array}{c}\theta\\\tau\end{array}\right)-
\left(\begin{array}{c}0\\t_0\end{array}\right).
\end{equation}
The mapping (\ref{Phi}) is defined on the set of couples
\begin{equation}
    \label{init_set_Phi}
    (\theta,\tau),\ 0<\tau<\theta<1,
\end{equation}
and it maps this set to itself. Through direct calculations, we confirm that the mapping (\ref{Phi}) has a unique fixed point (\ref{fix_point}), corresponding to the periodic solution (\ref{y0}).

Let's prove that for all other values of $\theta$ and $\tau$, $0<\tau<\theta<1$, different from $\tilde{\tau}$ and $\tilde{\theta}$ respectively, the $x_0$ regime is established after the transient process. Consider the pair of numbers $\Phi(\theta,\tau)=(\Phi_1(\theta,\tau),\Phi_2(\theta,\tau))$. These can be used as new $\theta$ and $\tau$ values, determining the initial set (\ref{init_set6}). If the numbers $\Phi_1(\theta,\tau)$, $\Phi_2(\theta,\tau)$ do not satisfy the inequality (\ref{neq_theta_tau}), then as previously shown, the $x_\varphi$ solution settles into the $x_0$ regime. If the numbers $\Phi_1(\theta,\tau)$, $\Phi_2(\theta,\tau)$ satisfy (\ref{neq_theta_tau}), i.e.,
\begin{equation}
    \label{neq_Phi}
    \frac{a\Phi_2(\theta,\tau)+1}{a+1}<\Phi_1(\theta,\tau)<\frac{(a+1)\Phi_2(\theta,\tau)+1}{a+1},
\end{equation}
then we can apply the mapping $\Phi$ again and repeat the reasoning. Thus, we are interested in the validity of the inequality
\begin{equation}
    \label{neq_Phi_iter}
    \frac{a\Phi_2^m(\theta,\tau)+1}{a+1}<\Phi_1^m(\theta,\tau)<\frac{(a+1)\Phi_2^m(\theta,\tau)+1}{a+1}
\end{equation}
for any natural number $m$.

The iterations of (\ref{Phi}) are described by the formulas 
\begin{equation}
    \label{Phi_oddIter}
    \Phi^{2k+1}\left(\begin{array}{c}\theta\\\tau\end{array}\right)=
(-1)^kT_0^k(a+1)
\left(\begin{array}{cc}1&-1\\t_0&-1\end{array}\right)
\left(\begin{array}{c}
\theta-\tilde{\theta}
\\
\tau-\tilde{\tau}
\end{array}\right)
+
\left(\begin{array}{c}
\tilde{\theta}
\\
\tilde{\tau}
\end{array}\right),\ k=0,1,\ldots,
\end{equation}
\begin{equation}
    \label{Phi_evenIter}
    \Phi^{2k}\left(\begin{array}{c}\theta\\\tau\end{array}\right)=
(-1)^kT_0^k
\left(\begin{array}{c}
\theta-\tilde{\theta}
\\
\tau-\tilde{\tau}
\end{array}\right)
+
\left(\begin{array}{c}
\tilde{\theta}
\\
\tilde{\tau}
\end{array}\right),\ k=1,2,\ldots.
\end{equation}
From the formulas (\ref{Phi_oddIter}) and (\ref{Phi_evenIter}), it follows that $\Phi^n$ also has a unique fixed point equal to $(\tilde{\theta},\tilde{\tau})$. To determine the mapping $\Phi^n$, it is necessary for the inequality (\ref{neq_Phi_iter}) to hold for $m=1,\ldots,n-1$. Additionally, from the formulas (\ref{Phi_oddIter}) and (\ref{Phi_evenIter}), it can be concluded that: 
\begin{equation}
    \label{Phi1_oddIter}
    \Phi^{2k+1}_1(\theta,\tau)=
(-1)^kT_0^k(a+1)(\theta-\tilde{\theta}-
\tau+\tilde{\tau})+
\tilde{\theta},\ k=0,1,\ldots,
\end{equation}
\begin{equation}
    \label{Phi2_oddIter}
    \Phi^{2k+1}_2(\theta,\tau)=
(-1)^kT_0^k(a+1)\big(t_0(\theta-\tilde{\theta})-
\tau+\tilde{\tau}\big)+
\tilde{\tau},\ k=0,1,\ldots,
\end{equation}
\begin{equation}
    \label{Phi1_evenIter}
    \Phi^{2k}_1(\theta,\tau)=
(-1)^kT_0^k(\theta-\tilde{\theta})+
\tilde{\theta},\ k=1,2,\ldots,
\end{equation}
\begin{equation}
    \label{Phi2_evenIter}
    \Phi^{2k}_2(\theta,\tau)=
(-1)^kT_0^k(
\tau-\tilde{\tau})+
\tilde{\tau},\ k=1,2,\ldots.
\end{equation}
Then, after transformations, the inequality (\ref{neq_Phi_iter}) takes one of the four forms depending on the residue $m$ modulo $4$:
\begin{equation}
    \label{neq_Phi_1}
     \tilde{\theta}-\frac{a(\tilde{\tau}-\tilde{\theta}+1)}{(a+1)T_0^{2l}}<
    \theta<
    \tilde{\theta}-\frac{a(\tau-\tilde{\tau})}{a+1}-
\frac{a(a\tilde{\tau}-(a+1)\tilde{\theta}+1)}{(a+1)^2T_0^{2l}}
\end{equation}
$\mbox{for } m=4l+1,\ l=0,1,\ldots;$
\begin{equation}
    \label{neq_Phi_2}
    \tilde{\theta}+\tau-\tilde{\tau}-
\frac{(a+1)(\tilde{\tau}-\tilde{\theta})+1}{{(a+1)T_0^{2l+1}}}<
    \theta<
    \tilde{\theta}+\frac{a(
\tau-\tilde{\tau})}{a+1}-
\frac{a\tilde{\tau}-(a+1)\tilde{\theta}+1}{(a+1)T_0^{2l+1}}
\end{equation}
$\mbox{for } m=4l+2,\  l=0,1,\ldots;$
\begin{equation}
    \label{neq_Phi_3}
    \tilde{\theta}-\frac{a(\tau-\tilde{\tau})}{a+1}+
\frac{a(a\tilde{\tau}-(a+1)\tilde{\theta}+1)}{(a+1)^2T_0^{2l+1}}<
    \theta<
    \tilde{\theta}+\frac{a(\tilde{\tau}-\tilde{\theta}+1)}{(a+1)T_0^{2l+1}}
\end{equation}
$\mbox{for } m=4l+3,\  l=0,1,\ldots;$
\begin{equation}
    \label{neq_Phi_4}
    \tilde{\theta}+\frac{a(
\tau-\tilde{\tau})}{a+1}+
\frac{a\tilde{\tau}-(a+1)\tilde{\theta}+1}{(a+1)T_0^{2l}}<
    \theta<
    \tilde{\theta}+\tau-\tilde{\tau}+
\frac{(a+1)(\tilde{\tau}-\tilde{\theta})+1}{{(a+1)T_0^{2l}}}
\end{equation}
$\mbox{for } m=4l,\  l=1,2,\ldots.$
When going to the limits in the inequalities (\ref{neq_Phi_1}), (\ref{neq_Phi_2}), (\ref{neq_Phi_3}), (\ref{neq_Phi_4}) as $l\to\infty$, we obtain the following system of inequalities:
\begin{eqnarray}
    \label{neq_Phi_lim}
     \tilde{\theta}\leq\theta\leq\tilde{\theta}-\frac{a(\tau-\tilde{\tau})}{a+1},
     &\quad&
    \tilde{\theta}+\tau-\tilde{\tau}\leq\theta\leq\tilde{\theta}+\frac{a(\tau-\tilde{\tau})}{a+1},
\nonumber\\
    \tilde{\theta}-\frac{a(\tau-\tilde{\tau})}{a+1}\leq\theta\leq\tilde{\theta},
    &\quad&
    \tilde{\theta}+\frac{a(\tau-\tilde{\tau})}{a+1}\leq\theta\leq\tilde{\theta}+\tau-\tilde{\tau},
\end{eqnarray}
which has a unique solution in the form of a point $(\tilde{\theta},\tilde{\tau})$.

Theorem \ref{th_2r_per} has been proven.
\end{proof}

\section{Instability of the periodic solution of equation (\ref{eq_rele}) with a ``short'' period}\label{sec_2root}

For the convenience of further analysis, instead of the initial function set (\ref{init_set6}), we introduce a set on a segment separated from zero. Fix the parameter $\sigma\in(0,\tilde{\tau})$, which will play the role of a neighborhood of zero. Let's consider the set
 \begin{eqnarray}
    \label{init_set7}
    \varphi\in C[-1-\sigma,-\sigma],\ \varphi(t)<0&\mbox{ as }& t\in[-1-\sigma,-\tilde{\theta}),\ \varphi(-\tilde{\theta})=0,\nonumber
    \\
    \varphi(t)>0&\mbox{ as }& t\in(-\tilde{\theta},-\tilde{\tau}),\ \varphi(-\tilde{\tau})=0,\nonumber
    \\
    \varphi(t)<0&\mbox{ as }& t\in(-\tilde{\tau},-\sigma),\ \varphi(-\sigma)=-\sigma.
\end{eqnarray}
Translation: The solution of equation (\ref{eq_hutch_discr}) with the initial function (\ref{init_set7}) coincides with (\ref{y0}) for $t\geq-\sigma$.

\begin{theorem}\label{th_2r_per_stab}
The periodic solution (\ref{y0}) of equation (\ref{eq_rele}) with the initial function (\ref{init_set7}) is unstable.

\end{theorem}
\begin{proof}
Let's put $x(t)=y_0(t)+h(t)$ and substitute it into equation (\ref{eq_hutch_discr}). Then, the perturbation $h(t)$ is determined by the equation:
\begin{equation}
    \label{eq_h}
    \dot{h}=R(y_0(t-1)+h(t-1))-R(y_0(t-1)).
\end{equation}
We will consider perturbations $h$ that do not lead out of the class of functions with two zeros on the initial segment under consideration.

The study of the stability of the solution (\ref{y0}) boils down to finding and evaluating the multipliers of the delayed equation (\ref{eq_h}). As in the work \cite{GlyKolRoz2013}, by multipliers of the differential-difference equation with delay, we understand the eigenvalue of the infinite-dimensional monodromy operator. To determine the monodromy operator, we fix the real numbers $\gamma_0,$ $\xi_\theta,$ $\xi_\tau$. The first one denotes the value of the function $h(t)$ at zero, and the other two represent the deviations of the zeros of the function $y_0(t)+h(t)$ from the corresponding zeros of the function $y_0(t)$ on the interval $[-1-\sigma,-\sigma]$. We introduce a set of functions
 \begin{eqnarray}
    \label{init_h}
   E=\{\gamma\in C[-1-\sigma,-\sigma],&&\nonumber
    \\
    \ y_0(t)+\gamma(t)<0&\mbox{ as }& t\in[-1-\sigma,-\tilde{\theta}+\xi_\theta),\ \varphi(-\tilde{\theta}+\xi_\theta)=0,\nonumber
    \\
   y_0(t)+\gamma(t)>0&\mbox{ as }& t\in(-\tilde{\theta}+\xi_\theta,-\tilde{\tau}+\xi_\tau),\ \varphi(-\tilde{\tau}+\xi_\tau)=0,\nonumber
    \\
    y_0(t)+\gamma(t)<0&\mbox{ as }& t\in(-\tilde{\tau}+\xi_\tau,-\sigma),\ \gamma(-\sigma)=\gamma_0.
\end{eqnarray}
On the figure \ref{pic:y0_gamma}, the function $y_0(t)+\gamma(t)$ is shown in black for some function $\gamma\in E$.
\begin{figure}[h]
\centering\includegraphics[width=18pc]{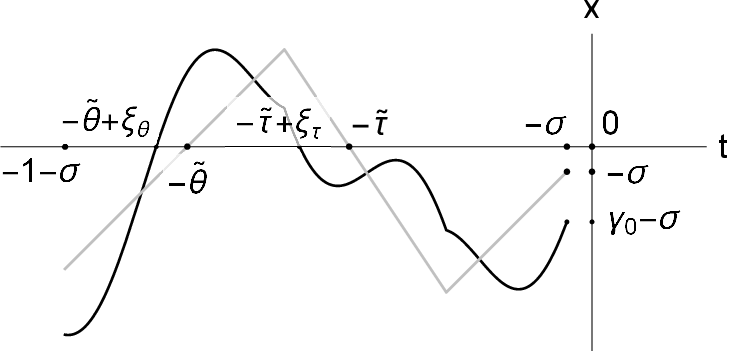}
 \caption{The gray curve is $y_0(t)$, the black curve is $y_0(t)+\gamma(t)$.}
\label{pic:y0_gamma}
\end{figure} 

An operator of monodromy for equation (\ref{eq_h}) is defined as an operator $M: E \to E$ that acts on any function $\gamma\in E$ according to the rule:
\begin{equation}
    M\gamma(t)=h_\gamma(t+\tilde{\theta}),\quad t\in[-1,0],
\end{equation}
where $h_\gamma(t)$ is the solution to equation (\ref{eq_h}) over the time interval $[-\sigma,\tilde{\theta}]$ with the initial function $\gamma\in E$. Let's prove that (\ref{eq_h}) has a unit multiplier and two multipliers whose absolute values exceed one.

Let's find the solution $h_\gamma(t)$ of equation (\ref{eq_h}).

On intervals where the signs of the functions $y_0(t-1)+\gamma(t-1)$ and $y_0(t-1)$ match, the right-hand side of (\ref{eq_h}) equals zero, then $h(t)$ is constant in these regions. In cases where the functions $y_0(t-1)+\gamma(t-1)$ and $y_0(t-1)$ have different signs, the right-hand side of (\ref{eq_h}) equals $a+1$ or $-a-1$, depending on the sign of the deviations $\xi_\theta$, $\xi_\tau$. Thus, in these regions, $h(t)$ becomes an affine function with the corresponding coefficient. Therefore,
$$
        h_\gamma(t)=\left\lbrace\begin{array}{cl}
\gamma_0,& t\in[-\sigma,1-\tilde{\theta}+\min\{0,\xi_\theta\}],\\
\gamma_0+(a+1)\xi_\theta,& t\in[1-\tilde{\theta}+\max\{0,\xi_\theta\},1-\tilde{\tau}+\min\{0,\xi_\tau\}],\\
\gamma_0+(a+1)(\xi_\theta-\xi_\tau),& t\in [1-\tilde{\tau}+\max\{0,\xi_\tau\},1+\min\{0,-\gamma_0\}].
\end{array}\right.
$$
We assume that $\gamma_0 < \sigma$ and the parameters $\xi_\theta$, $\xi_\tau$ are small enough so that the points $\tilde{\theta} - \tilde{\tau}$ and $\tilde{\tau}$ fall within the intervals $[1-\tilde{\theta}+\max\{0,\xi_\theta\},1-\tilde{\tau}+\min\{0,\xi_\tau\}]$ and $[1-\tilde{\tau}+\max\{0,\xi_\tau\},1+\max\{0,-\gamma_0\}]$ respectively. Then the equation $y_0(t) + h(t) = 0$ in a neighborhood of zero takes the form $t + \gamma_0 = 0$, and in the vicinity of $\tilde{\theta} - \tilde{\tau}$, it takes the form $-a(t-\tilde{\theta}+\tilde{\tau})+\gamma_0+(a+1)\xi_\theta=0$.
From this, we can deduce that the deviations of the zeros of the function $y_0(t) + h(t)$ from $0$ and $\tilde{\theta} - \tilde{\tau}$ are respectively $-\gamma_0$ and $\frac{\gamma_0}{a} + t_0\xi_{\theta}$. The deviation of the function $y_0(t) + h(t)$ at the point $\tilde{\theta}$ from the value $y_0(\tilde{\theta})$ is equal to $\gamma_0 + (a+1)\xi_\theta - (a+1)\xi_\tau$. Hence, it follows that a three-dimensional linear operator 
$$M_3\left(\begin{array}{c}\gamma_0,\\ \xi_\theta,\\ \xi_\tau\end{array}\right)=
\left(\begin{array}{c}\gamma_0+(a+1)\xi_\theta-(a+1)\xi_\tau\\-\gamma_0\\\frac{\gamma_0}{a}+t_0\xi_{\theta}\end{array}\right)$$
 is separated from the operator $M$. 
Its matrix: $$M_3=\left(\begin{array}{ccc}1&a+1&-(a+1)\\-1&0&0\\\frac1a&t_0&0\end{array}\right).$$ 

The characteristic equation has the form $(\mu-1)(\mu^2+T_0)=0$, from which we find the eigenvalues $\mu_1=1,$ $\mu_{2,3}=\pm\sqrt{T_0}i$. Since $\mu_2$ and $\mu_3$ lie outside the unit circle, we conclude the instability of the solution $y_0(t)$ of equation~(\ref{eq_rele}). 

Theorem \ref{th_2r_per_stab} is proved.
\end{proof}

It is worth noting that in the work \cite{Kasch_hutch}, where a different version of the Hutchinson equation was considered, rapidly oscillating solutions with periods shorter than the delay length were also found to be unstable.

\section{Conclusion}

The equation (\ref{eq_rele}) is considered, representing a relay generalized Hutchinson equation after an exponential substitution. For this equation, all possible solutions are constructed with an initial function from the set (\ref{init_set6}), containing exactly two zeros on the interval $[-1,0)$ denoted by $-\theta$, $-\tau$ or less than two. It is proven that for $\theta$ and $\tau$ equal to (\ref{fix_point}), the equation possesses a periodic unstable solution (\ref{y0}), with a period shorter than $1$, i.e., the delay of equation (\ref{eq_rele}). Additionally, it is demonstrated that for all $0<\tau<\theta<1$, different from (\ref{fix_point}), the system settles into the regime $x_0$, defined by formulas (\ref{x0}) after a finite time.

An interesting natural extension of this study would involve the classification of solutions of equation (\ref{eq_rele}) with an initial function having an arbitrary fixed number of zeros on the interval $[-1,0]$. However, this task comes with some technical complexity.


The work was carried out in part (problem statement, section 2, section 4, Preobrazhenskaya M. M.) at the expense of the Russian Science Foundation grant No. 22-11-00209, https://rscf.ru/project/22-11-00209/; partially (section 3, Preobrazhensky I. E.) within the framework of the implementation of the development program of the regional scientific and educational mathematical center (YarSU) with financial support from the Ministry of Science and Higher Education of the Russian Federation (Agreement on the provision of subsidies from the federal budget No. 075-02-2023-948).

%
%

\end{document}